\documentclass[12pt]{article}%
\usepackage[applemac]{inputenc}
\usepackage{amsmath,amssymb,fullpage}
\usepackage{amsthm}
\usepackage{amsfonts}
\usepackage{amsmath,amssymb}
\usepackage[applemac]{inputenc}
\usepackage{amsmath,amssymb,fullpage}
\usepackage{color}
\usepackage{color, colortbl}
\usepackage{amsmath}
\usepackage{amssymb}
\usepackage{graphicx}
\usepackage{tikz}
\usepackage{enumerate}
\usepackage{geometry}
\usetikzlibrary{arrows}
\usepackage{amsfonts}
\usepackage{amsmath,amssymb}
\usepackage[applemac]{inputenc}
\usepackage{amsmath,amssymb,fullpage}
\usepackage{color}
\usepackage{amsmath}
\usepackage{amssymb}
\usepackage{graphicx}
\usepackage{tikz}
\usepackage{graphicx}
\graphicspath{ {./images/} }
\usepackage{amsmath}
\usepackage{algpseudocode}
\usepackage{algorithm}
\usepackage{float}
\usepackage{mathabx}
\usepackage{siunitx}
\usepackage{placeins}

\let\Oldsection\section
\renewcommand{\section}{\FloatBarrier\Oldsection}

\let\Oldsubsection\subsection
\renewcommand{\subsection}{\FloatBarrier\Oldsubsection}

\let\Oldsubsubsection\subsubsection
\renewcommand{\subsubsection}{\FloatBarrier\Oldsubsubsection}

\usepackage{multirow}
\usepackage{caption}
\usepackage{subcaption}
\usepackage{bbm, dsfont}

\algdef{SE}[SUBALG]{Indent}{EndIndent}{}{\algorithmicend\ }%
\algtext*{Indent}
\algtext*{EndIndent}

\newtheorem{theorem}{Theorem}[section]
\newtheorem{lemma}[theorem]{Lemma}
\newcommand{\E}{E}
\newtheorem{proposition}[theorem]{Proposition}

\newtheorem{definition}[theorem]{Definition}

\newtheorem{example}[theorem]{Example}
\newtheorem{assumption}[theorem]{Assumption}

\usepackage{tikz}
\usetikzlibrary{shapes.geometric, arrows}

\tikzstyle{input} = [circle, minimum width=1cm, text centered, draw=black, fill=green!20]

\tikzstyle{output} = [circle, minimum width=1cm, text centered, draw=black, fill=blue!20]

\tikzstyle{lstm} = [rectangle, rounded corners, minimum width=2cm, minimum height=1cm,text centered, draw=black, fill=red!20]

 \tikzstyle{lin} = [rectangle, minimum width=2cm, minimum height=1cm,text centered, draw=black, fill=orange!20]

 \tikzstyle{act} = [ellipse, minimum width=2cm, minimum height=1cm,text centered, draw=black, fill=yellow!20]

\tikzstyle{dot} = [rectangle, minimum width=2cm, minimum height=1cm,text centered]

\tikzstyle{arrow} = [thick,->,>=stealth]
\tikzstyle{map} = [thick, dashed,->,>=stealth]

\usetikzlibrary{arrows,shapes,automata,petri}

\numberwithin{equation}{section}

\def\RB{\mathbb{R}}

\def\FC{\mathcal{F}}

\definecolor{LightCyan}{rgb}{0.88,1,1}

\def\1B{\text{1\!\!I}}

\def\RB{\mathbb{R}}

\def\FC{\mathcal{F}}

\definecolor{LightCyan}{rgb}{0.88,1,1}

\def\1B{\text{1\!\!I}}

\def\<{\langle}
\def\>{\rangle}

\def\E{\mathbb{E}}

\def\bX{\mathbf{X}}

\def\bS{\mathbf{S}}
\def\bbS{\mathbb{S}}
\def\dsqcup{\sqcup\mathchoice{\mkern-7mu}{\mkern-7mu}{\mkern-3.2mu}{\mkern-3.8mu}\sqcup}

\DeclareMathOperator*{\esssup}{ess\,sup}

\begin{document}


\title{SIG-BSDE for Dynamic Risk Measures}
\author{Nacira Agram $^{1}$, Jan Rems $^{2}$ and Emanuela Rosazza Gianin $^{3}$}
\date{\today}
\maketitle

\footnotetext[1]{Department of Mathematics, KTH Royal Institute of Technology 100 44, Stockholm, Sweden. \newline
Email: nacira@kth.se. Work supported by the Swedish Research Council grant (2020-04697) and the
Slovenian Research and Innovation Agency, research core funding No.P1-0448.}

\footnotetext[2]{Department of Mathematics, University of Ljubljana, Ljubljana, Slovenia. \newline Email: jan.rems@fmf.uni-lj.si. Work supported by Slovenian Research and Innovation Agency, research core funding No.P1-0448.}

\footnotetext[3]{Department of Statistics and Quantitative Methods, University of Milano-Bicocca, Milano, Italy. Email:  emanuela.rosazza1@unimib.it. Member of GNAMPA, INdAM, Italy. Work supported by Gnampa Research Project 2024 (PRR-20231026-073916-203).}


\begin{abstract}
In this paper, we consider dynamic risk measures induced by backward stochastic differential equations (BSDEs). We discuss different examples that come up in the literature, including the entropic risk measure and the risk measure arising from the ambiguous interest rate problem. We develop a numerical algorithm for solving a BSDE using the backward Euler-Maruyama scheme and the universal approximation theorem for the signature of a path. We prove the convergence theorem and use the algorithm to solve some examples of dynamic risk measures induced by BSDEs. At last a deep learning approach is included for solving the ambiguous interest rate problem as well. 
\end{abstract}


\section{Introduction}
BSDEs appear in their linear form as adjoint equations when dealing with stochastic control. Bismut \cite{b} was the first to study this type of linear BSDE. Subsequently, this theory has also been developed for the nonlinear case, notably in the seminal work by Pardoux and Peng \cite{PP}. Additional developments in BSDE theory and the first applications to finance can be found in the works of Pardoux \cite{P}, El Karoui, Peng, and Quenez \cite{EPQ}, while to risk measures in Rosazza Gianin \cite{RG} and Barrieu and El Karoui \cite{entropic-bsde}.

Given a driving Brownian motion $B$, a generator $f:\Omega\times\lbrack0,T]\times\mathbb{R}\times\mathbb{R}\rightarrow
\mathbb{R}$ and a terminal condition $\xi$ solving a BSDE consists of finding a process $(Y,Z)$
adapted to the considered filtration (the Brownian one), such that, at time $t$, $(Y,Z)$ satisfies the equation:
\begin{align*} 
dY_{t}=-f(t,Y_{t},Z_{t})dt+Z_{t}dB_{t};\quad
Y_{T}=X\in L^{2}(\mathcal{F}_{T}).
\end{align*}

It is often difficult to obtain an explicit solution for BSDEs. However, one can hope to obtain a numerical approximation. The backward Euler-Maruyama scheme for BSDEs was proposed in \cite{bem}. In \cite{bem-conv} and \cite{bem-conv2}, convergence results for the scheme were introduced. In \cite{bem-reg}, the authors rely on regression on function bases to estimate conditional expectations in the above mentioned schemes and provide convergence results. However, the choice of a specific function basis is quite arbitrary. In this paper, we propose a universal regression approach using signatures of paths from rough path theory. The main idea is to encode information about a path in a sequence of iterated integrals called a signature that enjoys nice algebraic properties. These properties allow us to estimate the conditional expectation as a linear function of the signature. For a more detailed view of signatures and rough paths, we suggest monographies \cite{sig-book} and \cite{friz-book}. 

Our method will be then applied to some dynamic risk measures induced by BSDEs. It is well known in the literatrure, indeed, the strong relationship between (time-consistent) dynamic risk measures and BSDEs. See, e.g., \cite{RG}, \cite{entropic-bsde} and \cite{delbaen-et-al}. Special attention will be devoted to examples of entropic risk measures and of risk measures dealing with ambiguity on interest rates.
Some recent papers address the problem of risk measures with deep/reinforcement learning (see, e.g., \cite{rm-rl} and \cite{fouque}). However, the methodology driven in these works is different from ours. Feng et al. \cite{fouque} consider a static setting for systemic risk measures, while Coache and Jaimungal \cite{rm-rl} deal with dynamic risk measures but not BSDE-induced.

In the recent years, much progress has been made in using deep learning approaches to find numerical solutions to BSDEs. In \cite{han}, the authors proposed a deep learning forward approach. The initial value $Y_0$ is initialized as a parameter of the model. Then, at each discrete time-point, the value of process $Z$ is estimated by a neural network and the whole trajectory of $Y$ is obtained through a forward Euler-Maruyama scheme. Parameters of the model are obtained in a way that minimizes the distance between the obtained and prescribed terminal value of solution $Y$. 

A different deep learning approach, backwards in time, was proposed in \cite{pham}. There, authors use the approach inspired by dynamic programming, where the terminal value of $Y$ is initialized, and then both $Y$ and $Z$ are approximated by a neural network at the previous time-step. When the optimal approximation at the penultimate time-step is obtained, following the minimization of error present in (one step) Euler-Maruyama scheme, the outcome is used as a target value for obtaining $Y$ and $Z$ at one time before the penultimate time. This is then iterated backwards until the algorithm reaches time $0$. This makes the approach local in time, which makes it suitable for finding numerical solutions for problems with larger terminal times. By relying on neural networks, both of these approaches are suitable for solving high-dimensional problems. However, the training times are relatively long, and memory requirements when a fine time partition is used are impractical. Finally, let us refer to \cite{rm-rl}, where authors use reinforcement learning for solving dynamic risk measures through their dual representation. 
\medskip

The paper is organized as follows. In Section \ref{sec: prel}, we recall notions, definitions, and properties of BSDEs and dynamic risk measures, and we also provide some examples on which we will focus later on. In Section \ref{sec: signature}, we present our method based on signatures. Specifically, we introduce a new algorithm for numerical solutions of BSDE, called SIG-BSDE, and prove some convergence results. In Section 4, we consider the case of ambiguity in interest rates and solve the problem using deep learning. In Section 5, we  discuss the performance of SIG-BSDE with respect to the deep BSDE algorithm in some particular cases such as the linear BSDE, the entropic risk measure, and the risk measure accounting for ambiguity on interest rates.

\section{Dynamic risk measures induced by BSDEs}\label{sec: prel}
Dynamic risk measures indeed extend the concept of static risk measures to a dynamic setting, allowing for risk assessments to change over time based on new information. BSDEs can be effectively utilized to define these dynamic risk measures.

Given a BSDE with a terminal condition  (often $\xi =-X$ where $X$ is a future payoff or loss), the solution $Y_t$ of the BSDE can be interpreted as the dynamic risk measure at time $t$ for the future payoff $X$. Formally, we write: $\rho_t (X) =Y_t,$ where the dynamic risk measures, $(\rho_t)_{t\in [0,T]},$ adapt to the information available up to time $t$. 
\subsection{BSDEs}
Let us recall the main result of Pardoux and
Peng's work of 1990 in \cite{PP}. For this, we shall introduce the
corresponding spaces in which the solution processes live. For the whole Section we will refer to \cite{PP}.

We consider a complete probability space $\left(  \Omega,\mathcal{F},P\right)
$ equipped with a one-\newline dimensional standard Brownian motion $B$. We
denote by $\mathbb{F}=\left(  \mathcal{F}_{t}\right)  _{t\geq0}$ the right-continuous complete filtration generated by $B.$ 

Let $S_{\mathbb{F}}^{2}$ denote the space of all real-valued continuous $\mathbb{F}%
$-adapted processes $Y=(Y_{t})_{t\in\lbrack0,T]}$ associated with the norm
$$\mathbb{E}\left[  \sup_{t\in\lbrack0,T]}|Y_{t}|^{2}\right]  <+\infty.$$ By
$H_{\mathbb{F}}^{2}$ we denote the space of all $\mathbb{F}$-adapted processes
$Z=(Z_{t})_{t\in\lbrack0,T]}$ which are square integrable over the space
$\Omega\times\lbrack0,T]$ equipped with the norm $$\mathbb{E}\left[  \int_{0}^{T}|Z_{t}
|^{2}dt\right]  <+\infty.$$
A solution of a BSDE with terminal time $T$, terminal condition $X$ and driver $f$ consists of a pair of processes $(Y,Z)$. The BSDE is typically written as:
\begin{equation}\label{eq1} 
dY_{t}=-f(t,Y_{t},Z_{t})dt+Z_{t}dB_{t};\quad
Y_{T}=X.
\end{equation}
or, in its integral form of equation \eqref{eq1}, 
\begin{align*} 
Y_t=X+\int_t^T f(t,Y_s,Z_s)ds-\int_t^T Z_sdB_{s}.
\end{align*}
Hence, taking conditional expectation, we get 
$$
Y_t=\mathbb{E} \left[ X +\int_t^T  f(t,Y_s,Z_s)ds \mid \FC_t \right].
$$
We recall Pardoux and Peng's existence and uniqueness result.
\begin{theorem}[Existence and Uniqueness] Let $X\in L^{2}(\mathcal{F}_{T}) $
and $f:\Omega\times\lbrack0,T]\times\mathbb{R}\times\mathbb{R}\rightarrow
\mathbb{R}$ a jointly measurable mapping satisfying the following assumptions:
\begin{itemize}
    \item $f(\cdot,\cdot,0,0)\in H_{\mathbb{F}}^{2},$
    \item $f(t,\omega,\cdot,\cdot)$ is uniformly Lipschitz, $dt\otimes dP$-a.e., i.e., there is some constant $L\in \mathbb{R}$ such that, $dP\otimes dt$-a.e., for all $y,y^{\prime},z,z^{\prime}\in\mathbb{R},$
    $$|f(t,\omega,y,z)-f(t,\omega,y',z')|\le L(|y-y'|+|z-z'|).$$
\end{itemize}
Then the BSDE \eqref{eq1} possesses a unique solution $(Y,Z)\in S_{\mathbb{F}}^{2}\times H_{\mathbb{F}}^{2}.$
\end{theorem}

We recall that the solution of a linear BSDE can be written as a conditional expectation via an exponential martingale (closed form).\\
Let $\alpha_t$ and $\beta_t$ be an $\mathbb{R}$-valued, $\mathbb{F}$-adapted processes. For each $t \in [0,T]$, $(\Gamma_{t,s})_{s\in [t,T]}$ be the unique solution of the following forward SDE
  \begin{equation}\label{eq2.20}
 \displaystyle  d\Gamma_{t,s} = \Gamma_{t,s}[ \alpha_s ds + \beta_s dB_s 
];\quad
 \Gamma_{t,t} = 1.
 \end{equation}

 \begin{theorem}[Closed formula]\label{thm:linear}
 Let $\alpha$ and $\beta$ be bounded adapted processes, $X \in L^{2}(\mathcal{F}_{T})$ and $\varphi$ adapted with $\displaystyle \mathbb{E} \left[ \int_0^T \varphi^2_t dt\right] < \infty$. Let $\Gamma$ be the so-called adjoint process defined as the solution of the SDE \eqref{eq2.20}. 
 Then the unique solution $(Y,Z)$ of the linear BSDE
 \begin{align}
 \label{eq2.18}
 dY_t  \displaystyle = - \left[ \varphi_t + \alpha_t Y_t + \beta_t Z_t
 \right]dt + Z_t dB_t 
  ;\quad
  Y_T =X,
  \end{align}
is given by
  \begin{equation}\label{eq2.19}
  Y_t = \mathbb{E} \left[ \Gamma_{t,T} X + \int_t^T \Gamma_{t,s} \varphi_s ds \mid \FC_t \right],\quad  0 \leq t \leq T, \quad \text{a.s.}
  \end{equation}
  \end{theorem} 
\begin{example}
\begin{itemize}
    \item If $\alpha_t= \beta_t=\varphi_t=0$, then the component of the solution Y given in equation \eqref{eq2.19} satisfies
    \begin{equation*}
  Y_t = \mathbb{E}\left[ X \mid \FC_t \right] \; , \; 0 \leq t \leq T, \quad \text{a.s.}
  \end{equation*}

    \item   If $\alpha_t=\varphi_t=0$, then the equations \eqref{eq2.18} takes the form
    
 \begin{align*}
 dY_t  \displaystyle = -  \beta_t Z_tdt + Z_t dB_t 
  ;\quad
  Y_T =X.
  \end{align*} 
To this end, we define the probability measure $Q$ by
$$
dQ=M_tdP \text{ on } \mathcal{F}_T,
$$
with
$$
M_t=\exp \Big(\int_0^t \beta_sdB_s-\frac{1}{2}\int_0^t \beta^2_sds\Big).
$$
Then 
\begin{align*}
 dY_t  \displaystyle = Z_t dB^Q_t 
  ,\quad
  Y_T =X,
  \end{align*} 
where
$$
B^Q_t=B_t-\int_0^t \beta_s ds.
$$

Therefore, the component of the solution Y given in equation \eqref{eq2.19} satisfies
\begin{equation*}
Y_t = \mathbb{E}_Q\left[ X \mid \FC_t \right] \; , \; 0 \leq t \leq T, \quad \text{a.s.}
\end{equation*}
\end{itemize}
\end{example}
A relevant BSDE for applications to Mathematical Finance and to dynamic risk measures is the one with a quadratic driver, which will be discussed in Section \ref{sec:entropic}. This is known as the entropic risk measure.

We shall recall one of the most important properties of BSDEs, which is the comparison theorem. It plays a crucial role in the study of optimization problems expressed in terms of BSDEs and corresponding dynamic measures of risk.
\begin{theorem}[Comparison Theorem]
  Let $X^1, X^2\in L^{2}(\mathcal{F}_{T})$, let $f^1$ and $f^2$ be two drivers of BSDEs. For $i=1,2$, let $(Y^i,Z^i)$ be a solution in $S_{\mathbb{F}}^{2}\times H_{\mathbb{F}
}^{2}$ of the BSDE
$$
-dY_{t}^i=f^i(t,Y^i_{t},Z^i_{t})dt-Z^i_{t}dB_{t}; \quad Y^i_{T}=X_i.
$$
Assume that 
\begin{align*}
X^1 \geq X^2 \quad  a.s. \text{ and } f^1(t,Y^2_t,Z^2_t)\geq f^2(t,Y^2_t,Z^2_t),\quad t\in[0,T],\quad dP\otimes dt\quad a.s.
\end{align*}
Then $ Y^1_t \geq Y^2_t$ a.s. for all $t\in [0,T]$.
\end{theorem}


\subsection{Dynamic risk measures}

Risk measures have been introduced in the literature as a tool to quantify the riskiness of financial positions, both in a static and in a dynamic framework. To be more concrete, let $T$ be a fixed finite time horizon and let $L^p\left(  \mathcal{F}_{T}\right)$, with $p \in [2,\infty]$, be the space of $\mathcal{F}_T$-random variables that are $p$-integrable. A \textit{static risk measure} is a map $\rho:L^{p}\left(  \mathcal{F}_{T}\right)
\rightarrow \mathbb{R}$ that quantifies now the riskiness of any $X \in L^{p}\left(\mathcal{F}_T \right)$ that can be seen as the profit and loss (or the return) of a financial position. A \textit{dynamic risk measure}, instead, measures the riskiness of ($\mathcal{F}_T$-measurable) positions at any time $t \in [0,T]$.

Here below we only recall the essential definitions and arguments on static and dynamic risk measures. We refer, among many others, to \cite{FS02}, \cite{FS}, \cite{FRG}, \cite{delb}, \cite{RG}, \cite{entropic-bsde} and to the references therein for a detailed treatment. 

\begin{definition}



A \emph{dynamic convex risk measure} is a family $(\rho_t)_{t\in [0,T]}$ such that
$\rho_t:L^{p}\left(  \mathcal{F}_{T}\right)
\rightarrow L^{p}\left(  \mathcal{F}_{t}\right)$ and satisfying, for any $t \in [0,T]$,
\begin{itemize}
\item (Convexity) $\rho_t(\lambda X+(1-\lambda)Z)\leq
\lambda\rho_t(X)+(1-\lambda)\rho_t(Z)$ for all $\lambda
\in[0,1]$ and all $X,Z \in L^{p}\left(
\mathcal{F}_{T}\right)  $.

\item (Monotonicity) If $X \leq Z,$ then $\rho_t(X)\geq\rho_t(Z)$.

\item (Translation invariance) $\rho_t(X+a_t)=\rho_t(X)-a_t$ for all
$X\in L^{p}\left(  \mathcal{F}_{T}\right)  $ and all  $a_t \in L^{p}\left(  \mathcal{F}_{t}\right)  $.

\item $\rho_t(0)=0$.
\end{itemize}
A \emph{static convex risk measure} is a static risk measure satisfying the static versions of the properties above. 
\end{definition}

The last axiom is known as a normalisation and usually it is assumed for convenience. Note that convexity is an important criterion for risk measures because it means that diversification reduces the risk. Working with non-convex risk measures can lead to counter-intuitive results. Nevertheless, the classical risk measure Value at Risk (VaR) is not convex, but still frequently used in practice.\vspace{0.5cm}

Dynamic risk measures can be constructed from BSDEs as follows: Let $f(t,z)$ be a convex function of $z$ ($f$ does not depend on $y$) with $f(t,0)=0$ for any $t \in [0,T]$. For given $X \in L^{2}(\mathcal{F}_{T})$, let $(Y_t , Z_t)$ be the solution of the BSDE
\begin{equation}\label{eq2.11}
d Y_t = - f(t, Z_t)dt+Z_tdB_t; \quad
Y_T= -X.
\end{equation}
Then
\begin{equation}\label{eq2.12}
\rho_t(X) := Y_t
\end{equation}
defines a dynamic convex risk measure. Furthermore, any dynamic risk measure induced by a BSDE as before naturally satisfies the following time-consistency property:
\begin{equation}
\rho_t(X)=\rho_t(-\rho_u(X)), \quad \text{for any } X \in L^{2}(\mathcal{F}_{T}),\quad t\leq u \leq T. 
\end{equation}
For a discussion on different formulation of time-consistency for risk measures and on the normalization assumption, we refer to \cite{di-nunno-rg} and to the references therein.
\subsubsection{Entropic risk measures} \label{sec:entropic}
A relevant example of a dynamic risk measure induced by a quadratic BSDE is the entropic risk measure.
It is a type of risk measure which depends on the risk aversion of the user through the exponential utility function. Let $\theta > 0$ and let $X \in L^p(\FC_T)$ for $p \geq 2.$ The entropic risk measure is defined as 
\begin{align*}
    \rho^{entr}(X) = \frac{1}{\theta} \log \left( \E \left[e^{-\theta X} \right] \right).
\end{align*}
The name entropic becomes more clear when one considers a dual representation where we have 
\begin{align*}
    \rho^{entr}(X) = \sup_{Q \in \mathcal{M}} \left\{ \E_Q\left[ -X \right] - \frac{1}{\theta} H(Q|P) \right\},
\end{align*}
where $H(Q|P) = \E \left[ \frac{dQ}{dP} \log \frac{dQ}{dP} \right]$ is the relative entropy of $Q \ll P$.
See \cite{FRG} and \cite{entropic-bsde} for a deeper discussion on this type of risk measure.

The dynamic version of the measure is obtained in a standard way as 
\begin{align} \label{eq: entropic-explicit}
    \rho^{entr}_t(X) = \frac{1}{\theta} \log \left( \E \left[e^{-\theta X} \mid \FC_t \right] \right).
\end{align}
In \cite{entropic-bsde} authors show that $\rho^{ent}_t(X) = Y_t$, where $Y$ is the first component of the solution of the quadratic BSDE
\[
dY_{t}=-\frac{\theta}{2}Z^2_t dt+ Z_{t}dB_{t};\quad 
Y_{T}=-X.
\]

By Proposition 4 in Detlefsen and Scandolo \cite{detlef-scandolo}, the entropic risk measure solving the previous equation has the following dual representation
\begin{equation} \label{eq: dual-entropic}
\rho_t^{entr}(X)= \esssup_{Q \ll P: \, Q=P \mbox{ on } \mathcal{F}_t} \{\E_Q[-X | \mathcal{F}_t] - c^{entr}_{t}(Q)\},
\end{equation}
with (minimal) penalty term
\begin{equation}\label{eq: penalty-entropic}
c^{entr}_{t}(Q)= \frac{1}{\theta} \E\left[ \left. \frac{dQ}{dP} \ln \left(\frac{dQ}{dP} \right) \right| \mathcal{F}_t\right].
\end{equation}

\begin{example}\label{ex:entrop}
    Let us consider $X = B_T$ in the preceding definitions. We get 

\begin{align*}
    \rho^{entr}_t(B_T) &= \frac{1}{\theta} \log \left( \E \left[e^{-\theta B_T} \mid \FC_t \right] \right) \\
    &= \frac{1}{\theta} \log \left( \E \left[e^{-\theta (B_T -B_t + B_t)} \mid \FC_t \right] \right)  \\
    &= \frac{1}{\theta} \log \left( e^{-\theta B_t}e^{\frac{\theta^2(T-t)} {2}}\right) \\
    &= -B_t + \frac{\theta(T-t)}{2}.
\end{align*}
This presents a benchmark solution to which we can compare the results of two numerical approaches later.
\end{example}

\subsubsection{Other particular cases of dynamic risk measures induced by BSDEs}

So far, we have focused on translation invariant (or cash-additive) risk measures. In the recent literature, cash-additivity has been discussed and sometimes replaced by cash-subadditivity:
\begin{equation} \label{eq: csa}
\rho_t(X+a_t) \geq \rho_t(X)-a_t \quad \mbox{ for all } X\in L^{p}\left(  \mathcal{F}_{T}\right),\quad a_t \in L^{p}_+\left(  \mathcal{F}_{t}\right),
\end{equation}
also to deal with stochastic interest rates or ambiguity on interest rates. See El Karoui and Ravanelli \cite{ELK-ravanelli} for a discussion on this subject and for impact on the corresponding BSDEs.
\medskip

Related to the discussion above, a dynamic risk measure induced by a BSDE of a form different from \eqref{eq2.11} is the one consider in Example 7.2 of El Karoui and Ravanelli \cite{ELK-ravanelli}, with the following driver with depends on $y$ but not in $z$
\begin{equation*}
f(t,y)=-\beta_t y,
\end{equation*}
with $(\beta_t)_{t \in [0,T]}$ being an adapted stochastic process. As discussed in \cite{ELK-ravanelli}, the corresponding dynamic risk measure is
\begin{equation} \label{eq: dynamic-csa-1}
 \rho_t(X)= \E \left[ \left. e^{-\int_t^T \beta_s ds} (-X)\right| \mathcal{F}_t \right]   
\end{equation}
and satisfies monotonicity, convexity, cash-subadditivity and $\rho_t(0)=0$.  Moreover, it satisfies
$$
\rho_{tu} (-\rho_{uT}(X))=\rho_{tT}(X), \quad \mbox{ for any } X \in L^{\infty}(\mathcal{F}_T),\quad 0 \leq t \leq u \leq T,
$$
but not $$\rho_{tT} (-\rho_{uT}(X))=\rho_{tT}(X).$$

\begin{example}
Let's assume that the process $(\beta_t)_{t \in [0,T]}$ satisfies a Cox-Ingersoll-Ross (CIR) model to describe the evolution of interest rates. The forward process is given by
$$
d\beta_t = a(b-\beta_t)dt + \sigma\sqrt{\beta_t}dB_t; \qquad \beta_0 = x_0 > 0,
$$
for $a,b,\sigma >0.$ The couple of processes $Y$ and $Z$ satisfy 
\begin{align}\label{eq:cir}
    Y_t = 1- \int_t^T \beta_sY_s ds - \int_t^T Z_s dB_s.
\end{align}
In \cite{cox}, authors derive the explicit solution given by 
\begin{align*}
    Y_t = \Bigg( \frac{2 \gamma \exp\big(\frac{(\gamma + a)(T-t)}{2}\big)}{2\gamma+(\gamma + a)\exp(\gamma(T-t))} \Bigg)^{2ab/\sigma^2} \exp \Big(\frac{2(1-\exp(\gamma(T-t)))\beta_t}{2\gamma+(\gamma + a)\exp(\gamma(T-t))}\Big),
\end{align*}
with $\gamma = \sqrt{a^2 + 2 \sigma^2}.$
\end{example}

More in general, always in Example 7.2 of \cite{ELK-ravanelli}, the following dynamic risk measure satisfying monotonicity, convexity, cash-subadditivity 
\begin{equation} \label{eq: dynamic-csa-2}
 \rho_t(X)= \esssup_{0 \leq r_t \leq \beta_t \leq R_t} \E \left[ \left. e^{-\int_t^T \beta_s ds} (-X)\right| \mathcal{F}_t \right]   
\end{equation}
for some fixed, bounded and adapted processes $(r_t)_{t \in [0,T]}$ and $(R_t)_{t \in [0,T]}$, is also induced by a BSDE with driver
\begin{align}\label{eq:g-sup}
f(t,y)=\sup_{0 \leq r_t \leq \beta_t \leq R_t} (-\beta_t y).    
\end{align}
\bigskip

More in general, risk measures induced by BSDEs with convex drivers $g(t,y,z)$ that are decreasing in $y$ are shown to be cash-subadditive in \cite{ELK-ravanelli}. In the same paper, a dual representation for risk measures induced by BSDEs with convex driver $g(t,y,z)$ that is decreasing in $y$ is also established. 

The risk measures discussed in this Section will be used as benchmark to investigate the performance of the method on signature discussed in the following Section.

\section{Signatures and BSDE} \label{sec: signature}
In this section, we recall some notions from rough path theory. In particular, we use signatures in the Euler-Maruyama scheme for the numerical solution of BSDEs. 
\subsection{Signatures}
In this section, we briefly introduce the notions of rough paths and signatures of paths. In particular, we state the universal approximation theorem for continuous paths introduced in \cite{uat} that we use in the deep learning algorithm. This section is a very brief overview of the rough path tools needed to state this theorem. For a more general overview of rough path theory, we suggest \cite{sig-book} and \cite{friz-book}.\\

Let us denote by $T(\RB^d) = \prod_{k=0}^\infty (\RB^d)^{\otimes k}$ the extended tensor algebra equipped with standard addition, tensor multiplication and scalar multiplication. For $N \in \mathbb{N}$ denote the truncated tensor algebra $T(\RB^d) = \bigoplus_{k=0}^N (\RB^d)^{\otimes k}$. By $\mathcal{C}([0,T],\RB^d)$ we denote a space of continuous paths from $[0,T]$ to $\RB^d$ endowed with uniform topology. 

Consider a partition $\pi_{0,T} = \{0=t_0 < \cdots < t_n=T\}$ of an interval $[0,T]$. For a path $S \in \mathcal{C}([0,T],\RB^d)$ its $p$-variation is defined as 
\begin{equation*}
    \lVert S \rVert_{p-var} = \sup_{\pi_{0,T}}  \left( \sum_{i=0}^{n-1}  \lVert S_{t_{i+1}}- S_{t_i}\rVert^p  \right)^{1/p}.
\end{equation*}
By $\mathcal{C}^p([0,T],\RB^d)$ we denote a subspace of $\mathcal{C}([0,T],\RB^d)$ consisting of paths with finite $p$-variation.

Now, we can present a definition of a signature of a continuous semimartingale that plays a crucial role later on. 

\begin{definition}
\label{def:sig}
    Let $(S_t)_{t \in [0,T]} \in \mathcal{C}^p([0,T],\RB^d)$ be a continuous semimartingale with $p > 2.$ The signature $Sig(S)_{s,t}$ of $S$ on an interval $[s,t] \subseteq [0,T]$ is an element of $T(\RB^d)$ defined by $Sig(S)_{s,t} = (1,\bbS ^1_{s,t},\ldots, \bbS^k_{s,t}, \ldots )$, where
    \begin{equation*}
        \bbS^k_{s,t} = \int_{s < t_1 < \cdots <t_k <t} \circ dS_{t_1} \otimes \cdots \otimes dS_{t_k}.
    \end{equation*}
    We denote by $Sig_{s,t}^N(S) = (1,\bbS^1_{s,t}, \ldots, \bbS^N_{s,t})$ the truncated signature of depth $N$. It has dimension $\frac{d^{N+1} - 1}{d-1}$.
\end{definition}

In the above definition, the integrals are well-defined and understood in the Stratonovich sense. Signatures admit some useful algebraic properties. The most important one is the Chen's relation. It states that for $s < u < t$ it holds 
\begin{align*}
    Sig(S)_{s,t} = Sig(S)_{s,u} \otimes Sig(S)_{u,t}.
\end{align*}
Another algebraic property is related to the following notion of Shuffle products. 

\begin{definition}
        Let $I = \{i_1, \ldots, i_n \}$ and $J = \{j_1, \ldots, j_m\}$. The shuffle product is defined as 
        $$
        e_I \dsqcup e_J = (e_{I'} \dsqcup e_J)\otimes e_{i_n} + (e_I \dsqcup e_{J'})\otimes e_{j_m}, 
        $$
        where $e_I \dsqcup e_{\emptyset} = e_{\emptyset} \dsqcup e_I = e_I.$
    \end{definition}

We call $\mathbf{a} \in T(\RB^d)$ group like if 
    $$
    \langle  \mathbf{b_1}, \mathbf{a} \rangle \langle  \mathbf{b_2}, \mathbf{a} \rangle = \langle \mathbf{b_1} \dsqcup \mathbf{b_2}, \mathbf{a} \rangle,
    $$
    for all $\mathbf{b_1}, \mathbf{b_2} \in T(\mathbb{R}^d)$ and we write $\mathbf{a} \in G(\RB^d).$

This sheds light on the reason for using Stratonovich integral in Definition \ref{def:sig}. Indeed, if one considers a Brownian motion and wants to define its signature using It\^{o} integral, it is not hard to see that it is not group-like. 

The importance of group-like property is later revealed in Theorem \ref{thm:uat}. To state it, we must, however, view a special example of a signature of a semimartingale in a more general setting of Rough path theory.

Let us denote $\Delta_T = \{(s,t) \in [0,T]^2 \mid s \leq t \}$. The following object with values in the truncated tensor algebra is paramount. 

\begin{definition}
        Let $\bS : \Delta_T \to T^N(\RB^d)$ be a continuous map with values 
        $$
        \bS_{s,t} = (\bS^0_{s,t}, \ldots, \bS^N_{s,t}).
        $$
        We call $\bS$ a multiplicative functional of degree $N$ if $\bS^0_{s,t} \equiv 1$ and $\bS_{s,t} = \bS_{s,u} \otimes \bS_{u,t}$, for $s \leq u \leq t.$
    \end{definition}

We see that multiplicative functionals admit the algebraic Chen's equality. We need to endow them with analytical properties as well. This is done by defining control functions that bound their $p$-variation. A control function is a continuous non-negative function $w : \Delta_T \to [0,\infty)$ which is super-additive in a sense that
\begin{align*}
    w(s,u) + w(u,t) \leq w(s,t), 
\end{align*}
for all $s<u<t$ in $[0,T]$.

We say that a multiplicative functional $\bS$ of degree $N$ has a finite $p$-variation controlled by $w$ if 

\begin{align}\label{eq:control}
    \lVert \bS^k_{s,t} \rVert \leq \frac{w(s,t)^{k/p}}{\beta \Gamma(i/p +1)}, \qquad k=1, \ldots , N, \quad (s,t) \in \Delta_T,
\end{align}
where $\beta$ is a constant depending only on $p$.

This additional analytical regularity gives us the extension theorem that is a key building block of rough path theory. The proof can be found in \cite{sig-book}.

 \begin{theorem}\label{thm:extension}
        Let $p \geq 1$ and $\bS$ be a multiplicative functional of degree $N$ with finite $p$-variation, where $\lfloor p \rfloor \leq N.$ Denote by $pr_{\leq N}$ projection on first $N$ components of tensor algebra. Then for every $n > N$ there exists a unique continuous multiplicative functional $\mathbf{R}$ of degree $n$ such that 
        \begin{enumerate}
            \item $pr_{\leq N}(\mathbf{R}) = \bS $
            \item $\mathbf{R}$ has finite $p$-variation. 
        \end{enumerate}
    \end{theorem}

Now, we can finally define a rough path. 

\begin{definition}
    Let $p \geq 1$. We call $\bS$ a $p$-rough path if $\bS$ is a multiplicative functional of degree $\lfloor p  \rfloor$ with finite $p$-variation. We denote the space of such objects $\Omega_T^p(\RB^d).$ Furthermore if $\bS$ takes values in $G^{\lfloor p \rfloor}(\RB^d)$ we call it a weakly-geometric $p$-rough path and denote their space $\mathcal{W}\Omega_T^p(\RB^d).$
\end{definition}

By lifting the path to level two, we obtain an additional structure that allows one to construct pathwise integrals and, through that, differential equations called rough differential equations.

Observe that a truncated signature $Sig^{\lfloor p \rfloor}(S)_{s,t}$ from Definition \ref{def:sig} is a weakly-geometric rough path, while the whole signature $Sig(S)_{s,t}$ is its unique extension given by Theorem \ref{thm:extension}. 

We can now present the two main results that will be relevant in the following part of this paper. Generally, a signature does not determine the path uniquely but only up to a tree-like equivalence. However, this technical restriction can be omitted if we restrict ourselves to the paths with at least one strictly increasing component. As we explain later, this assumption actually comes naturally in our setting. The proof of the following proposition and a more detailed view of the uniqueness of signatures can be found in \cite{sig-unique}.

\begin{proposition}
    Let $S$ and $R$ be continuous $\RB^d$-valued semimartingales and denote with $\widehat{S} = (t,S_t)$ and $\widehat{R} = (t,R_t)$ their augmentations with time. Then $$Sig(\widehat{S})_T = S(\widehat{R})_T \text{ if and only if } S = R.$$
\end{proposition}

Now, we present the universal approximation theorem, the proof of which can be found in \cite{uat}. 

\begin{theorem}\label{thm:uat}
        Let $K \subset \mathcal{W}\Omega_T^p(\RB^d)$ be a compact set and let $f: \mathcal{W}\Omega_T^p(\RB^d) \to \RB$ be a continuous functional. Then for any $\varepsilon > 0$ there exists some $\ell \in T(\RB^d)$ such that 
        $$
        \sup_{\bX \in K} |f(\bX) - \langle \ell , S(\bX)_{0,T} \rangle | < \varepsilon.
        $$
    \end{theorem}

The main takeaway from the theorem above is that any continuous functional of a weakly-geometric rough path can be approximated arbitrarily well by a linear function of a signature. The main idea of the proof is to check the needed properties to apply the Stone-Weierstrass theorem, which gives us the approximation part. The linear structure of this approximation is due to the fact that weakly-geometric rough paths are group-like. Then, the polynomial can be rewritten linearly as a shuffle product of higher terms of a signature.

 \medskip


Let $X \in L^p(\FC_T)$ for $p \geq 2$ and let us assume that $(\FC_t)_{t \geq 0}$ is generated by Brownian motion $B$. Furthermore, put $\mu_t = \E[ X | \FC_t ].$ Here, we present an algorithm for directly estimating conditional expectation $\mu_t$. This approach relies on the use of signatures of paths. In particular, we use Theorem \ref{thm:uat} that allows us to estimate $\mu_t$ as a function of the signature of time-extended Brownian motion $B$. Note that taking a time-extended Brownian motion also makes sense from a heuristic perspective since $\mu_t$ is also a function of time. Furthermore, this approximation function is linear, which enables us to use fast and traceable algorithms such as linear or ridge regression. The algorithm is presented below, where we work with time-discrete approximations of processes on time partition $\pi = \{0=t_1 < t_2 < \cdots < t_N =T\}$, while the signature is truncated at depth $D$.

\begin{algorithm}[H]
\caption{\textbf{CE-SIG} Algorithm}
\label{algo:cex}
\begin{algorithmic}

\Require Target value $\{X^{j}\}^{1 \leq j \leq M}$, discrete Brownian path $\{B^{j}_{i}\}^{1 \leq j \leq M}_{0 \leq i \leq k}$ and signature depth $D$.
\State Put $\{\hat{B}^{j}_{i}\}^{1 \leq j \leq M}_{0 \leq i \leq k} = \{(B^{j}_{i},t_k)\}^{1 \leq j \leq M}_{0 \leq i \leq k}$

    \State Compute the signatures $S^j = S_{0,t_k}^D(\hat{B}^j_{0:k})$ of depth $D$ on time interval $[0,t_k]$ for each $j=1, \ldots,M$
    \State Obtain optimal coefficients $c^j = \{c_l\}_{0 \leq l \leq D}$ by ridge regression on pairs $S^j$ and $X^j$ for $j=1, \ldots,M$
    \State Estimate $\mu_k^j = \langle c^j, [1,S^j] \rangle$ for each  $j=1, \ldots,M$    
\State \Return $\{Y_k^j\}^{1 \leq j \leq M}$
\end{algorithmic}
\end{algorithm}

Of course, we work with a truncated signature in the algorithm above. Justification for this approach is outlined in Lemma \ref{lemma}, where we discuss the convergence of our algorithm.

\subsection{SIG-BSDE Algorithm} 
In this section, we propose an algorithm for finding a numerical solution of a BSDE using a variant of the least square regression approach. The backward Euler-Maruyama scheme for BSDEs was proposed in \cite{bem}. In \cite{bem-conv} and \cite{bem-conv2}, convergence results for the scheme were introduced. In \cite{bem-reg}, the authors rely on regression on function bases to estimate conditional expectations in the above mentioned schemes and provide convergence results. However, the choice of a specific function basis is quite arbitrary. Here, we propose a more universal approach using signatures of paths. 

Consider now the following decoupled forward-backward system

\begin{align}\label{eq:fbsde}
    X_t &= x + \int_0^t b(s,Xs)ds + \int_0^t\sigma(s,X_s)dBs \\
    Y_t &= g(X_T) + \int_t^T f(s,X_s,Y_s,Z_s)ds - \int_t^TZ_sdB_s,
\end{align}
where, for simplicity of notation, all processes are assumed to be one-dimensional. 

We will work under the following:
\begin{assumption}\label{assumption}$ $\newline
     \vspace{-0.8cm}\begin{enumerate}[(i)]
    \item functions $b,\sigma,f$ and $g$ are deterministic and $b(\cdot,0),\sigma(\cdot,0),f(\cdot,0,0,0)$ and $g(0)$ are bounded. 
    \item  $b, \sigma,f$ and $g$ are uniformly Lipschitz continuous in $(x,y,z)$ with Lipschitz constant $L$. 
    \item  $b,\sigma$ and $f$ are uniformly H\"{o}lder-$\frac{1}{2}$ continuous in $t$ with H\"{o}lder constant $L$.
\end{enumerate}
\end{assumption}

Recall the time partition $\pi$ from the previous section. The forward process is approximated by a numerical solution on partition $\pi$ through the Euler-Maruyama scheme 
\begin{align}\label{eq:fem}
    X^N_{{k+1}} = X_{k}^N + b(t_k, X_{k}^N)\Delta_t + \sigma(t_k,X_k^N) \Delta B_{k+1}, \quad X^N_{0} = x, \quad k=0, \ldots, N-1.
\end{align}
Here $\Delta_t$ denotes the time increment in equidistant partition $\pi$ while $\Delta B_k$ denote independent $\mathcal{N}(0, \Delta_t)$ distributed random variables representing the Brownian motion increments. 

The backward Euler-Maruyama scheme is given by initializing $Y^N_N = g(X^N_N)$ and 
\begin{align}\label{eq:bem}
    Z^N_k = \frac{1}{\Delta_t} \E_k\left[ Y^N_{k+1} \Delta B_{k+1} \right], \quad Y^N_{k+1} = \E_k \left[ Y^N_{k+1} + f(t_k,X^N_k, Y^N_{k+1}, Z_k^N) \Delta_t \right], 
\end{align}
for $k=N-1, \ldots,0,$ where $\E_k$ denotes conditional expectation with respect to $\FC_{t_k}.$

The motivation for the above scheme is the following. We have 
\begin{align*}
    Y_{t_k} &= Y_{t_{k+1}} + \int_{t_k}^{t_{k+1}} f(t,X_t,Y_t,Z_t) - \int_{t_k}^{t_{k+1}} Z_t dB_t \\
    &\approx Y_{t_{k+1}} + f(t,X_{t_k},Y_{t_k},Z_{t_k}) \Delta_t - Z_{t_k} \Delta B_{k+1}.
\end{align*}
If we multiply the above expression 
with $\Delta B_{k+1}$ and apply conditional expectation at time $t_k$ to both sides of the equation, we obtain 
\begin{align*}
    Z_{t_k} = \frac{1}{\Delta_t} \E_k[Y_{t_{k+1}} \Delta B_{k+1} ].
\end{align*}
Furthermore, by the same argument and by substituting $Y_{t_k}$ with $Y_{t_{k+1}}$ in the driver $f$, we obtain 
\begin{align*}
    Y_{t_k} = \E_k[Y_{t_{k+1}} + f(t_k,X_{t_k},Y_{t_{k+1}},Z_{t_k})],
\end{align*}
which provides us with an explicit backward Euler-Maruyama scheme. 

Another possibility is to consider an implicit scheme where $Y_{t_k}$ is obtained as a solution of 

\begin{align*}
    Y_{t_k} = \E_k[Y_{t_{k+1}} + f(t_k,X_{t_k},Y_{t_{k}},Z_{t_k})],
\end{align*}
through Picard iteration. 

In the last step, we use the signature method to estimate conditional expectation in \eqref{eq:bem} and thus obtain the numerical solution, as described in the following algorithm.

\begin{algorithm}[H]
\caption{\textbf{SIG-BSDE} Algorithm}
\label{algo:SIG-BSDE}
\begin{algorithmic}

\Require Terminal values $\{Y^j_N \}^{1 \leq j \leq M}$, forward process $\{X^i_{n}\}^{1 \leq i \leq M}_{0 \leq n \leq N}$  Brownian motion samples $\{B^i_{n}\}^{1 \leq i \leq M}_{0 \leq n \leq N}$.
\For{$N-1 \leq k \leq 0$}
    \State Put $Z_k^j = \frac{1}{\Delta_t} \E \left[ Y_{k+1}^j (B^j_{k+1}-B_k^j) \mid \FC_{t_k} \right]$ using \textbf{CE-SIG}
    \State Put $Y_{k}^j = \E \left[ Y_{k+1}^j + f(t_k, X_k^j, Y_{k+1}^j, Z_k^j) \Delta_t \mid \FC_{t_k}  \right]$ using \textbf{CE-SIG}
\EndFor    
\State \Return BSDE solution $(Y,Z)$
\end{algorithmic}
\end{algorithm}

\subsection{Convergence results}

Here, we present a convergence result for the numerical scheme presented in the previous section. The convergence can be obtained following the next steps. First, we need to consider the convergence of forward and backward Euler-Maruyama schemes described in the previous section. Then, it remains to show that we can efficiently estimate conditional expectations in the backward scheme \eqref{eq:bem}, for which we rely on the universal approximation Theorem \ref{thm:uat} and properties of truncated signatures. 

We begin by stating the well-known convergence theorems for the forward and backwards Euler-Maruyama schemes. Proofs of both can be found in \cite{bsde-book}.

\begin{theorem}
    \label{thm:fem-conv}
    Let $X$ and $X^N$ be as in \eqref{eq:fbsde} and \eqref{eq:fem}, respectively. Suppose Assumption \ref{assumption} holds. Then there exists a constant $C>0$ such that 
    \begin{align*}
        \max_{0 \leq k \leq N-1} \E \left[ \sup_{t_k \leq t \leq t_{k+1}} |X_t-X^N_k|^2 \right] \leq C\left(1+|x|^2 \right)\Delta_t.
    \end{align*}
\end{theorem}

\begin{theorem}
    \label{thm:bem-conv}
    Let $Y$ and $Y^N$ be as in \eqref{eq:fbsde} and \eqref{eq:bem}, respectively. Suppose Assumption \ref{assumption} holds. Then, there exists a constant $C>0$ such that 
    \begin{align*}
        \max_{0 \leq k \leq N-1} \E \left[ \sup_{t_k \leq t \leq t_{k+1}} |Y_t-Y^N_k|^2 \right] + \sum_{k=0}^{N-1} \E 
\left[ \int_{t_k}^{t_{k+1}} |Z_t - Z_{t_k}^N |^2 dt  \right] \leq C\left(1+|x|^2 \right)\Delta_t.
    \end{align*}
\end{theorem}

To obtain the convergence it suffices to estimate the conditional expectations in \eqref{eq:bem} with truncated signatures arbitrarily well. We prove the following lemma.

\begin{lemma}\label{lemma}
    Let $\xi$ be an $\FC_t$-measurable random variable and let us denote $\mu_t = \E_t\left[\xi \right]$. Furthermore, let $K \subset \mathcal{W}\Omega^p_T(\RB^d)$ be a compact set, denote an augmented Brownian motion with $\widehat{B}_t = (t,B_t)$ and with $\widehat{\mathbf{B}}$ its geometric rough path. Then for any $\varepsilon > 0$, there exist a positive integer $D$ and $\ell \in T(\RB^2)$, such that a.s.
    \begin{align*}
        \sup_{t \in [0,T]} \sup_{\widehat{\mathbf{B}} \in K} |\mu_t - \langle \ell, Sig^D(\widehat{\mathbf{B}}_{0:t}) \rangle | < \varepsilon.
    \end{align*}
\end{lemma}

\begin{proof}
    We can view $\mu_t$ as $\mu(\widehat{\mathbf{B}})$, a continuous map on $\mathcal{W}\Omega^p_T(\RB^d)$. By Theorem \ref{thm:uat} we have 

    \begin{align*}
        \sup_{t \in [0,T]} \sup_{\widehat{\mathbf{B}} \in K} |\mu_t - \langle \ell, Sig(\widehat{\mathbf{B}}_{0:t}) \rangle | < \frac{\varepsilon}{2}.
    \end{align*}
Furthermore, we have $$|\langle \ell, Sig(\widehat{\mathbf{B}}_{0:t})-Sig^D(\widehat{\mathbf{B}}_{0:t}) \rangle | \leq \lVert \ell \rVert \lVert Sig(\widehat{\mathbf{B}}_{0:t})-Sig^D(\widehat{\mathbf{B}}_{0:t}) \rVert,$$ where we have, by the definition of the truncated signature $$\lVert Sig(\widehat{\mathbf{B}}_{0:t})-Sig^D(\widehat{\mathbf{B}}_{0:t}) \rVert = \sum_{i \geq D+1} \lVert \widehat{\mathbb{B}}^i_{0:t} \rVert.$$
By the compactness of $K$ and since $\widehat{\mathbf{B}} \in \mathcal{W}\Omega(\RB^2)$ with property in Equation \eqref{eq:control}, we have that the above sum admits a convergent uniform norm and goes to $0$ as $D \to \infty$. Then there exists $D \in \mathbb{N}$ such that 

    \begin{align*}
        \sup_{t \in [0,T]} \sup_{\widehat{\mathbf{B}} \in K} |\mu_t - \langle \ell, Sig^D(\widehat{\mathbf{B}}_{0:t}) \rangle | < \frac{\varepsilon}{2} + \sup_{t \in [0,T]} \sup_{\widehat{\mathbf{B}} \in K} \lVert Sig(\widehat{\mathbf{B}}_{0:t})-Sig^D(\widehat{\mathbf{B}}_{0:t}) \rVert < \varepsilon.
    \end{align*}    
\end{proof}
The coefficient $\ell$ that ensures the convergence in the lemma above can be obtained through the least square estimation by observing that $\mu_t$ is the unique $\FC_t$-measurable random variable that minimizes $\E \left[(\xi - \mu_t)^2 \right]$.

By the above lemma, we have 
\begin{align*}
   \E \left[(\xi - \mu_t)^2 \right] &=   \E \left[(\xi - \langle \ell, Sig^D(\widehat{\mathbf{B}}_{0:t}) \rangle + \langle \ell, Sig^D(\widehat{\mathbf{B}}_{0:t}) \rangle - \mu_t)^2 \right] \\
   &= \E \left[(\xi - \langle \ell, Sig^D(\widehat{\mathbf{B}}_{0:t}) \rangle)^2 \right] \\ 
   &+ \E \left[(\xi - \langle \ell, Sig^D(\widehat{\mathbf{B}}_{0:t}) \rangle)(\langle \ell, Sig^D(\widehat{\mathbf{B}}_{0:t}) \rangle - \mu_t) \right] \\
   &+ \E \left[(\langle \ell, Sig^D(\widehat{\mathbf{B}}_{0:t}) \rangle - \mu_t)^2 \right] \\
   &=  \E \left[(\xi -\langle \ell, Sig^D(\widehat{\mathbf{B}}_{0:t}) \rangle)^2 \right]+\varepsilon.
\end{align*}
By the law of large numbers, the expectation above is efficiently approximated by the Monte Carlo method using a sample of $M$ trajectories.

The reasoning above, together with Theorems \ref{thm:fem-conv}, \ref{thm:bem-conv} and Lemma \ref{lemma} yield the following convergence theorem. 

\begin{theorem}
    \label{thm:conv-main}
    Let $\widehat{Y}^{D,N,M}$ and $\widehat{Z}^{D,N,M}$ denote approximations of $Y^N$ and $Z^N$, respectively, using approach described in Algorithm \ref{algo:cex} with sample size $M$ and signature depth $D$. Then there exist $D,N,M \in \mathbb{N}$ such that 
    \begin{align*}
        \max_{0 \leq k \leq N-1} \E \left[ \sup_{t_k \leq t \leq t_{k+1}} |Y_t-Y^{D,N,M}_k|^2 \right] + \Delta_t \sum_{k=0}^{N-1} \E 
\left[ \int_{t_k}^{t_{k+1}} |Z_t - Z_{t_k}^{D,N,M} |^2 dt  \right] < \varepsilon.
    \end{align*}
\end{theorem}

\section{Deep learning for the ambiguous interest rates}
Here, we try to solve numerically the example of a dynamic risk measure dealing with ambiguity on interest rates described in Equation \eqref{eq: dynamic-csa-2}. Let us recall that we want to find the dynamic risk measure given by 
\begin{equation*}
 \rho_t(X)= \esssup_{0 \leq r_t \leq \beta_t \leq R_t} \E \left[ \left. e^{-\int_t^T \beta_s ds} (-X)\right| \mathcal{F}_t \right], 
\end{equation*}
induced by a BSDE.

Instead of tackling the "control" problem in \eqref{eq: dynamic-csa-2} directly, we would rather solve the corresponding BSDE with the driver 
\begin{equation}
    \label{eq:g}
    f(t,y)=\sup_{0 \leq r_t \leq \beta_t \leq R_t} (-\beta_t y).
\end{equation}

It is easy to check that the optimal $\beta$ is given as $\widehat{\beta} = R \mathds{1}(y < 0) + r \mathds{1}(y \geq 0)$. However, explicitly computing the dynamic risk measure under the optimal choice $\widehat{\beta}$ is not easy. To do so numerically, we choose the following deep learning approach to model the driver $g$ in Equation \eqref{eq:g}.

First, we use a feedforward neural network $\varphi:\RB^3 \to \RB$ with inputs $y,r, R$ to predict $\beta$ that maximizes $(-\beta y)$ where $r \leq \beta \leq R$. The neural network $\varphi$ has 3 hidden layers and 11 hidden neurons at each layer. Between the hidden layers, a ReLu activation function is used. The architecture of this neural network can be seen in Figure \ref{fig:nn}. For the loss function, the outputs are aggregated by taking the sample mean. During the training phase, the $y$ input is generated by a standard normal distribution, while the inputs $r$ and $R$ are chosen uniformly while satisfying $r \leq R$. This is done so that the approach applies also to the problems where $\widehat{\beta}$ of different forms is possible. Then we use the Algorithm \ref{algo:SIG-BSDE} where the driver is substituted by the neural network $\varphi$. We summarize this description using the following algorithm.

\begin{figure}[!htb]
    \centering
    \includegraphics[width=0.5\linewidth]{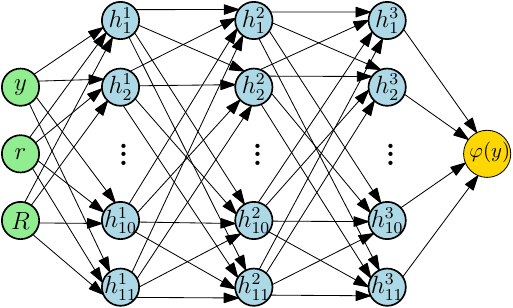}
    \caption{Architecture of the neural network $\varphi$}
    \label{fig:nn}
\end{figure}

\begin{algorithm}[H]
\caption{\textbf{AIR} Algorithm}
\label{algo:AIR-SIG}
\begin{algorithmic}

\Require Terminal values $\{X^j \}^{1 \leq j \leq M}$, Brownian motion samples $\{B^i_{n}\}^{1 \leq i \leq M}_{0 \leq n \leq N}$, a neural network $\varphi$, batch size $P$.
\For{$i$ in number of epochs}
    \State Initialize $\{r^p \}^{1 \leq p \leq P}, \{R^p \}^{1 \leq p \leq P}, \{y^p \}^{1 \leq p \leq P}$ from a training distribution. 
    \State Compute $\beta = \max\{r, \min\{R, \varphi(y,r,R)  \} \}$
    \State $loss = \E [ \beta y ]$
    \State Compute gradient of $loss$ and modify weights of $\varphi$ using backpropagation. 
\EndFor

\State Numerically solve the BSDE $d\rho_t = \varphi(\rho_t) \rho_t dt + Z_t dB_t$ using \textbf{SIG-BSDE} algorithm

\end{algorithmic}
\end{algorithm}

\section{Numerical results}

Here, we present numerical solutions for several examples of BSDEs using the above mentioned approaches. First, we test our algorithm on a linear BSDE with an explicit solution. Furthermore, we investigate the CIR example in Equation \eqref{eq:cir} and the entropic risk measure considered in Example \ref{ex:entrop}. Results are compared with the ones obtained by following the deep learning backward dynamic programming (DBDP) approach proposed in \cite{pham}. At last, we provide a numerical solution of the ambiguous interest rate example outlined in Equation \eqref{eq: dynamic-csa-2}. Apart from providing graphs of predicted and theoretical dynamics, we use estimated relative $L^2$ error given by 

\begin{align*}
    ERL^2(\hat{a},a) = \sqrt{\frac{1}{M}\sum_{j=1}^M \sum_{k=0}^N |\hat{a}_k^j - a_k^j|^2\Delta_t},
\end{align*}
for discrete paths $a$ and $\hat{a}$.

For the \textbf{SIG-BSDE} approach, unless stated otherwise, we use sample size $M=2^{13}$, number of discretization points $N=500$ and terminal time $T=1$. We truncate the signature at depth $D=3$, which means that the regression estimation described in \ref{algo:cex} works on 14 data points. The ridge regression is performed with the regularization parameter $\lambda=0.3$. Both depth and regularization parameter do not seem too important as long as the depth is greater than 1 to model nonlinearity and $\lambda$ is not too big. 

For the DBDP approach that we use as a comparison to our model, we use a different set of parameters due to computational restrictions. This is due to the use of deep learning methods instead of regression, which are much more time-consuming. We use $N=100$ timesteps and a batch size of $2^{10}$. A feed-forward neural network consists of 3 hidden layers with 11 neurons each and $tanh$ activation functions. Let us now look at particular examples. 

\subsection{Linear BSDE}
Here below we deal with the example of a linear BSDE in Equation \eqref{eq2.18}, where we consider as terminal condition $X = \exp(\beta B_T - \frac{\beta^2}{2}T)$. Using Theorem \ref{thm:linear} and the martingale property of $X$ we deduce that the explicit solution is given by
\begin{align*}
    Y_t = e^{\beta B_t - \frac{\beta^2}{2}t} + e^{\beta^2T} e^{2\beta B_t - 2 \beta^2 t} - e^{2\beta B_t - \beta^2 t}.
\end{align*}

In Figure \ref{fig:lin}, we compare two numerical solutions, one done with \textbf{SIG-BSDE} algorithm and the other with the DBDP approach. We also state the estimated $L^2$ errors.

 \begin{figure}[!htb]
\centering
\begin{subfigure}{.5\textwidth}
  \centering
  \includegraphics[width=1\linewidth]{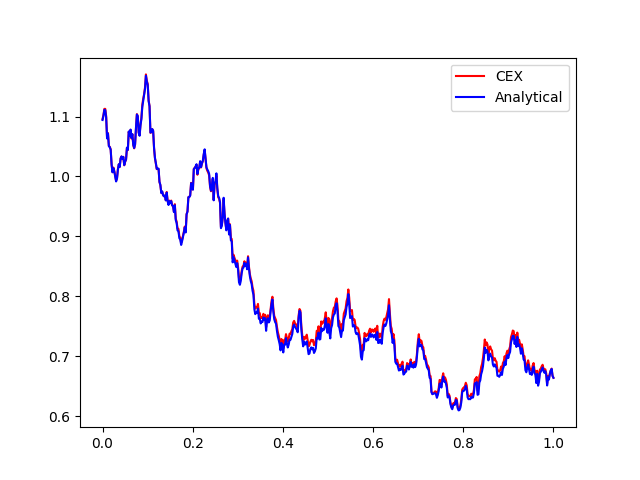}
  \caption{\textbf{SIG-BSDE}, $ERL^2 = 0.013$}
  \label{fig:lin-sig}
\end{subfigure}%
\begin{subfigure}{.5\textwidth}
  \centering
  \includegraphics[width=1\linewidth]{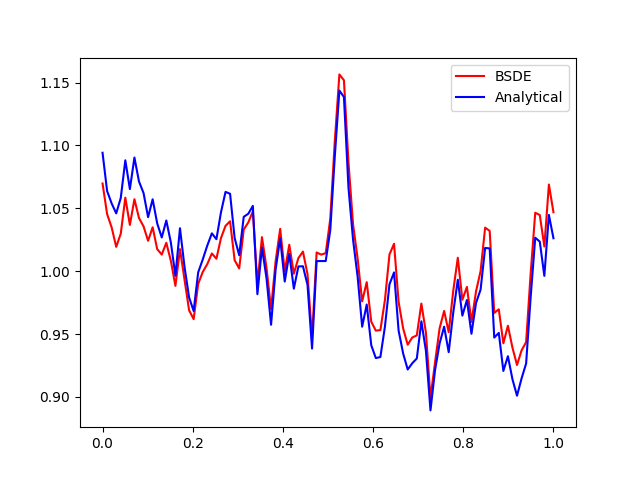}
  \caption{DBDP, $ERL^2 = 0.033$}
  \label{fig:lin-deep}
\end{subfigure}
    \caption{Linear BSDE: Comparison of a realization of one numerical solution obtained by \textbf{SIG-BSDE} and one by DBDP, with their respective estimated $L^2$ erorrs}
    \label{fig:lin}
\end{figure}

We can see that \textbf{SIG-BSDE} outperforms the DBDP approach. Furthermore, it provides more flexibility, where one can consider more finely refined time intervals due to smaller computational complexity. In Figure \ref{fig:lin-hist},  we can see a distribution of estimated $L^2$ errors over 50 independent iterations of the algorithm. 

\begin{figure}[!htb]
    \centering
    \includegraphics[width=0.5\linewidth]{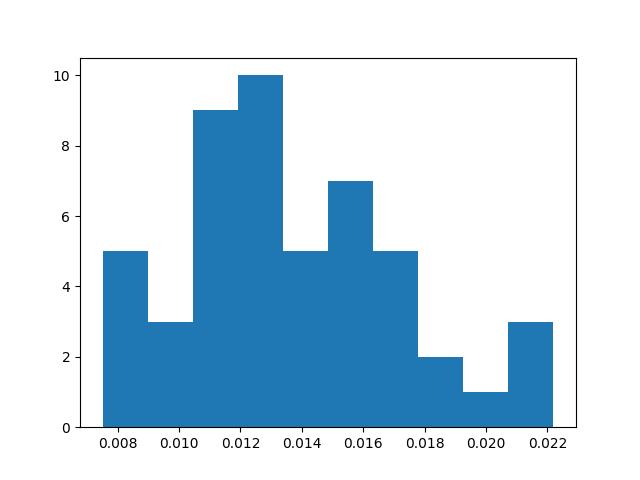}
    \caption{Linear BSDE: Histogram of estimated $L^2$ errors over 50 iterations. The average error is 0.013 with a standard deviation 0.004}
    \label{fig:lin-hist}
\end{figure}

At last, let us comment on the behaviour of the \textbf{SIG-BSDE} algorithm with respect to the sample size. In Figure \ref{fig:lin-m}, we can see that the error reduces with a rate proportional to $\frac{1}{\sqrt{M}}$, which is in accordance with the theory of regression models. For each sample size, 50 independent iterations were made, whose mean value is then presented as the final error. The same was found for all examples; hence, we do not include other graphs. 

\begin{figure}[!htb]
    \centering
    \includegraphics[width=0.5\linewidth]{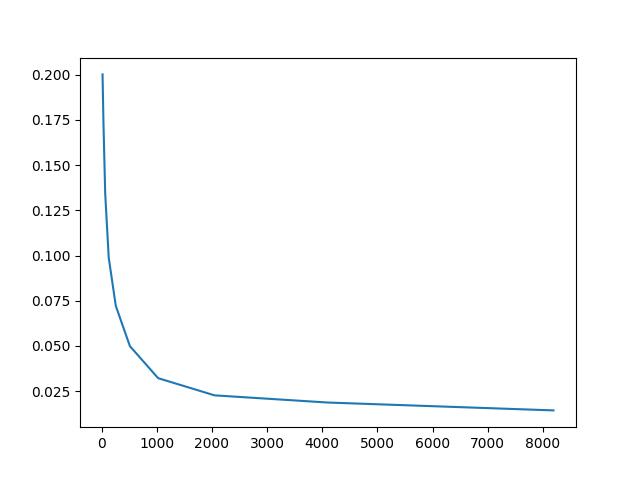}
    \caption{Linear BSDE: Estimated $L^2$ error with respect to different sample sizes $M$}
    \label{fig:lin-m}
\end{figure}

\subsection{Entropic risk measure}

We consider an example of the dynamic entropic risk measure presented in Section \ref{sec:entropic}. For simulations, we set parameter $\theta = 0.3$. We find the numerical representation by solving the corresponding BSDE.  
In Figure \ref{fig:ent}, we again compare the numerical solution given by \textbf{SIG-BSDE} algorithm to the one obtained via the DBDP approach, with the corresponding estimated $L^2$ errors.

 \begin{figure}[!htb]
\centering
\begin{subfigure}{.5\textwidth}
  \centering
  \includegraphics[width=1\linewidth]{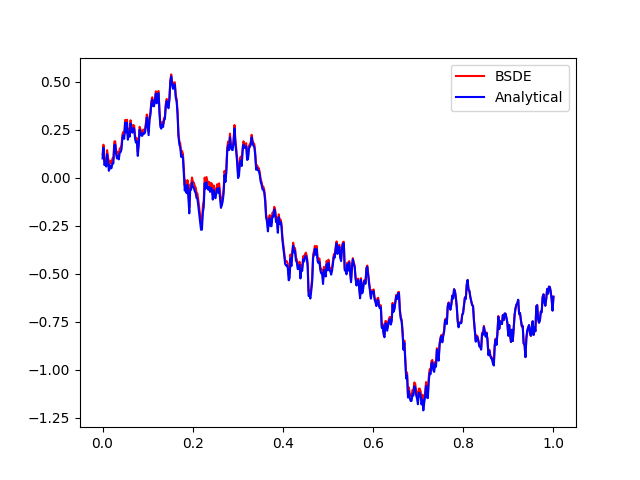}
  \caption{\textbf{SIG-BSDE}, $ERL^2 = 0.073$}
  \label{fig:ent-sig}
\end{subfigure}%
\begin{subfigure}{.5\textwidth}
  \centering
  \includegraphics[width=1\linewidth]{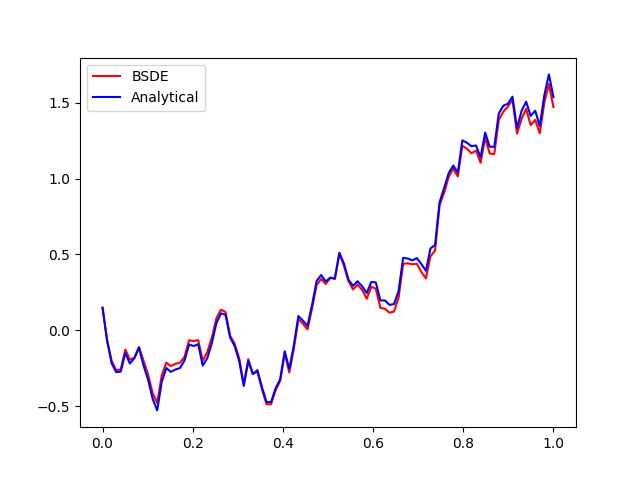}
  \caption{DBDP, $ERL^2 = 0.077$}
  \label{fig:ent-deep}
\end{subfigure}
    \caption{Entropic risk measure: Comparison of a realization of one numerical solution obtained by \textbf{SIG-BSDE} and one by DBDP, with their respective estimated $L^2$ errors}
    \label{fig:ent}
\end{figure}

Here, we can see that \textbf{SIG-BSDE} and the DBDP approach give comparable results. The distribution of estimated $L^2$ errors over 50 independent algorithm iterations can be found in Figure \ref{fig:ent-hist}. 

\begin{figure}[!htb]
    \centering
    \includegraphics[width=0.5\linewidth]{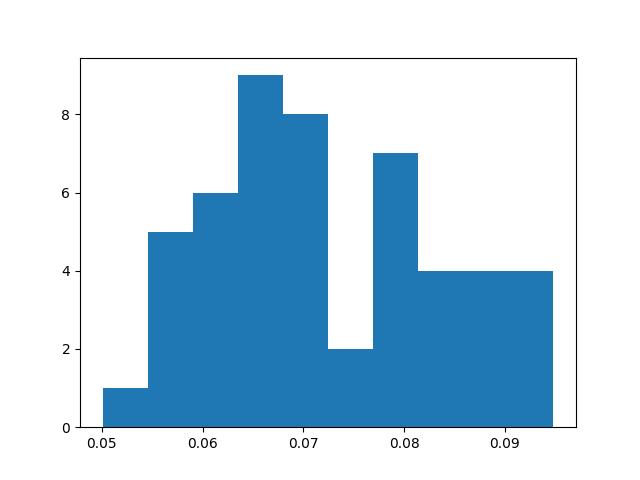}
    \caption{Entropic risk measure: Histogram of estimated $L^2$ errors over 50 iterations. The average error is 0.073 with a standard deviation 0.011}
    \label{fig:ent-hist}
\end{figure}

\subsection{The ambiguous interest rates}
Here, we try to find a numerical solution to the dynamic risk measure presented in \eqref{eq: dynamic-csa-2}. For the simplicity of the presentation, let us take $X = B_T, r_t=0$ and $R_t = 1$, where $B$ is a Brownian motion. For the neural network $\varphi$, we have two hidden layers with 11 neurons each. The convergence of the loss function and predicted output can be seen in Figures \ref{fig:losses} and \ref{fig:scatter}.

 \begin{figure}[!htb]
\centering
\begin{subfigure}{.5\textwidth}
  \centering
  \includegraphics[width=1\linewidth]{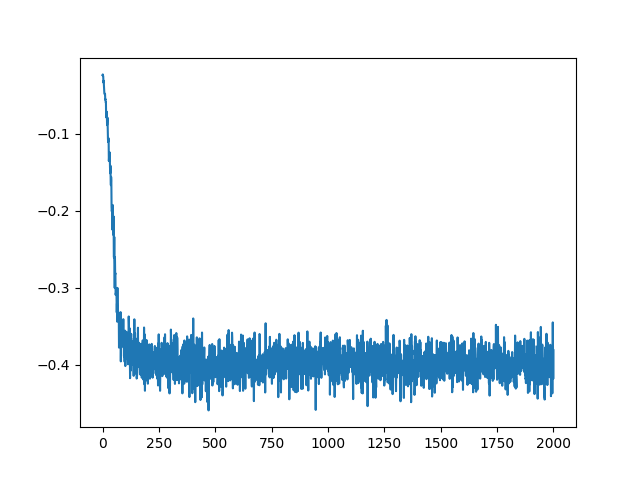}
  \caption{Convergence of the loss function}
  \label{fig:losses}
\end{subfigure}%
\begin{subfigure}{.5\textwidth}
  \centering
  \includegraphics[width=1\linewidth]{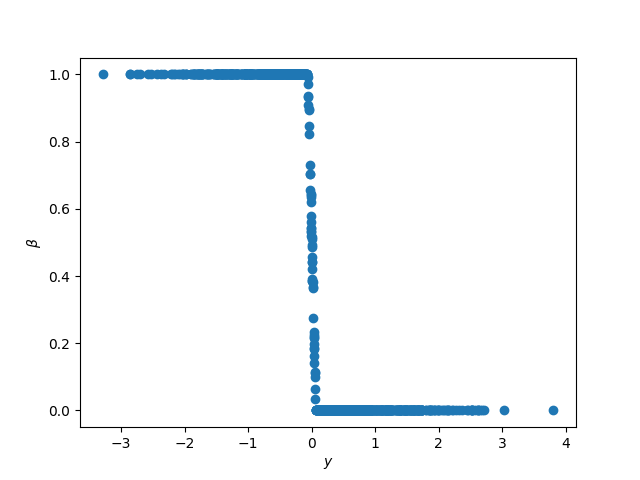}
  \caption{Predicted $\beta$ values given $y$}
  \label{fig:scatter}
\end{subfigure}
    \label{fig:beta}
    \caption{Deep learning estimation of the driver $f(t,y)$}
\end{figure}

Let us denote 
\begin{equation*}
\rho_t(X,\beta)=\E \left[ \left. e^{-\int_t^T \beta_s ds} (-X)\right| \mathcal{F}_t \right] 
\end{equation*}
for given $X$ and $\beta$. By construction $\rho_t(X) \geq \rho_t(X,\beta)$ for any choice of $\beta$.

Under deterministic and constant rate $\beta$ the dynamical risk measure becomes $\rho_t(X,\beta) = -e^{-\beta(T-t)}B_t$. In Figure \ref{fig:beta-compare} we plot a realization of $\rho_t(X)$ and $\rho_t(X,\beta)$ for choices $\beta \equiv 0,\frac{1}{2},1$.

\begin{figure}[!htb]
    \centering
    \includegraphics[width=0.5\linewidth]{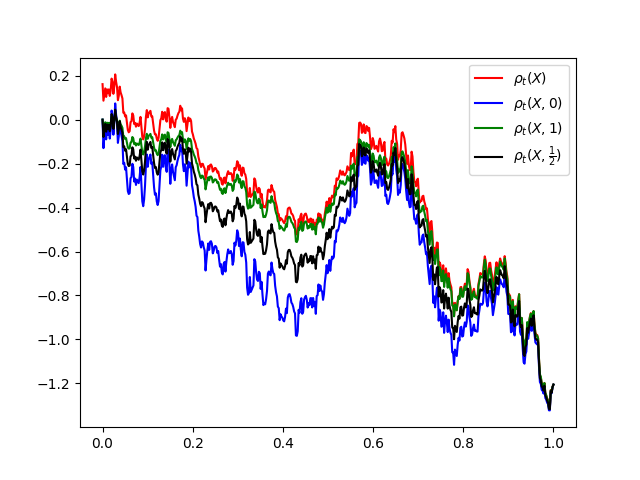}
    \caption{Comparison of $\rho_t(X)$ with $\rho_t(X,\beta)$ for $\beta \equiv 0,\frac{1}{2},1$}
    \label{fig:beta-compare}
\end{figure}

We can see that the simulated $\rho_t(X)$ is, in fact, always greater or equal to the $\rho_t(X,\beta)$ for the sub optimal choice of $\beta$. This is especially evident for times closer to 0, where the sub optimality of the choice of $\beta$ has time to manifest itself.

\subsubsection{CIR model}
We consider a CIR-modeled interest rate example from Equation \eqref{eq:cir}, with a choice of parameters $a = b = \sigma = 1$. In Figure \ref{fig:cir}, we see a comparison of the numerical solution given by \textbf{SIG-BSDE} algorithm to the one obtained via the DBDP approach. We also state the estimated $L^2$ errors.

 \begin{figure}[!htb]
\centering
\begin{subfigure}{.5\textwidth}
  \centering
  \includegraphics[width=1\linewidth]{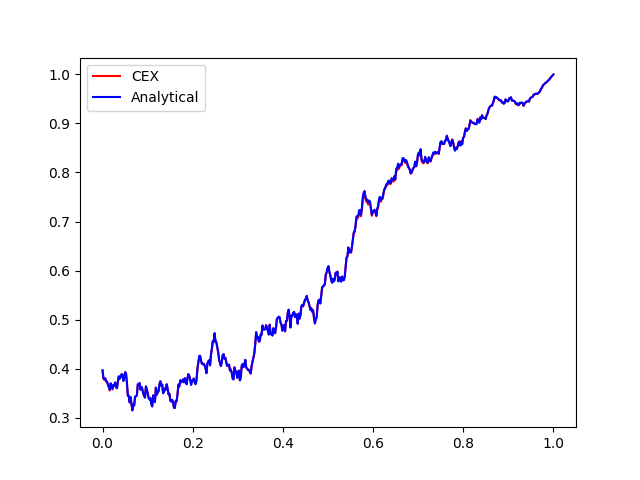}
  \caption{\textbf{SIG-BSDE}, $ERL^2 = 0.005$}
  \label{fig:cir-sig}
\end{subfigure}%
\begin{subfigure}{.5\textwidth}
  \centering
  \includegraphics[width=1\linewidth]{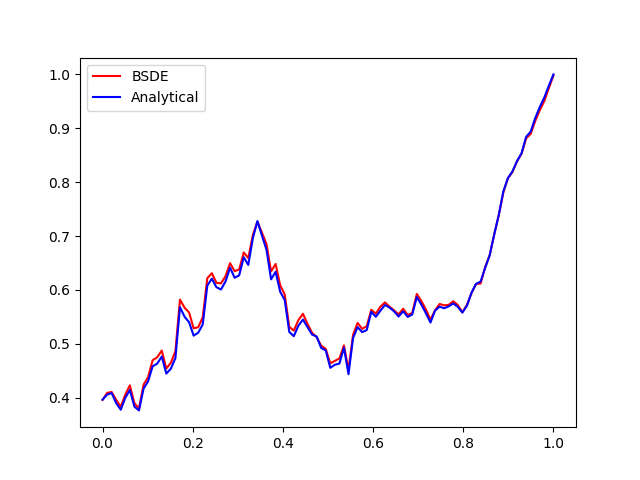}
  \caption{DBDP, $ERL^2 = 0.011$}
  \label{fig:cir-deep}
\end{subfigure}
    \caption{CIR risk measure: Comparison of a realization of one numerical solution obtained by \textbf{SIG-BSDE} and one by DBDP, with their respective estimated $L^2$ errors}
    \label{fig:cir}
\end{figure}

In this example as well, the \textbf{SIG-BSDE} outperforms the DBDP approach. In Figure \ref{fig:cir-hist}, we present a distribution of estimated $L^2$ errors over 50 independent algorithm iterations. 

\begin{figure}[!htb]
    \centering
    \includegraphics[width=0.5\linewidth]{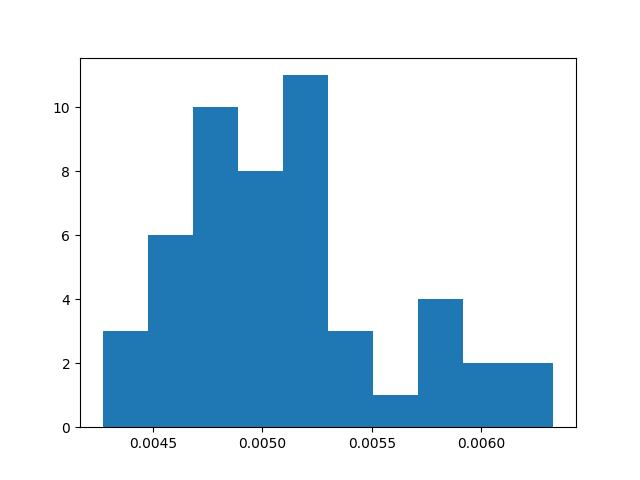}
    \caption{CIR risk measure: Histogram of estimated $L^2$ errors over 50 iterations. The average error is 0.005 with a standard deviation \num{4.8e-4}}
    \label{fig:cir-hist}
\end{figure}

\end{document}